\documentclass{amsart}

\usepackage{amsmath,amsfonts,amssymb,amsthm}
\usepackage[hidelinks]{hyperref}
\usepackage{mathtools}
 \usepackage{pdfpages}
\usepackage[capitalize]{cleveref}
\usepackage{stmaryrd}
\usepackage{tikz-cd}
\usepackage{mathabx}
\usepackage{graphicx}

\newtheorem{theorem}{Theorem}[section]

\newtheorem{proposition}[theorem]{Proposition}
\newtheorem{lemma}[theorem]{Lemma}
\newtheorem{corollary}[theorem]{Corollary}

\newtheorem{conjecture}[theorem]{Conjecture}

\theoremstyle{definition}

\newtheorem{example}[theorem]{Example}

\DeclareMathOperator{\depth}{depth}
\DeclareMathOperator{\roots}{roots}

\DeclareMathOperator{\supp}{supp}

\DeclareMathOperator{\deck}{Deck}

\DeclareMathOperator{\Ima}{im}

\newcommand{\Z}{\mathbb{Z}}

\newcommand{\N}{\mathbb{N}}

\newcommand{\Q}{\mathbb{Q}}

\newcommand{\isom}{\cong}
\newcommand{\normal}[1]{\langle\!\langle #1 \rangle\!\rangle}

\begin{document}

\title{Residually rationally solvable one-relator groups}


\author{Marco Linton}
\address{Instituto de Ciencias Matem\'aticas, CSIC-UAM-UC3M-UCM, Madrid, Spain}
\email{marco.linton@icmat.es}

\maketitle

\begin{abstract}
We show that the intersection of the rational derived series of a one-relator group is rationally perfect and is normally generated by a single element. As a corollary, we characterise precisely when a one-relator group is residually rationally solvable.
\end{abstract}

\section{Introduction}

If $G$ is a group, the subgroup
\[
G^{(i)}_{\Q} := \left\{g \bigm\vert g^k\in \left[G_{\Q}^{(i-1)}, G^{(i-1)}_{\Q}\right],\, k\neq 0\right\}\trianglelefteq G,
\]
is the $i^{\text{th}}$ term of the \emph{rational derived series}, with $G^{(0)}_{\Q} = G$. It is so denoted as $G^{(i)}_{\Q}$ is the kernel of the natural map to the first homology group of $G^{(i-1)}_{\Q}$ over the rationals. A group $G$ is \emph{rationally solvable} (or \emph{$\Q$-solvable}) if $G_{\Q}^{(n)} = 1$ for some $n$ and is \emph{residually rationally solvable} if
\[
G_{\Q}^{(\omega)} := \bigcap_{i=0}^{\infty}G_{\Q}^{(i)} = 1. 
\]
The (rational) derived series can be continued to form the transcendental (rational) derived series. The intersection of all terms of the (rational) transcendental derived series is the maximal (rationally) perfect subgroup of $G$.

Baumslag proved in \cite{Ba71} that all positive one-relator groups, that is, groups with a single defining relation that do not mention any inverse generators, are residually solvable; or in other words, they have $G^{(\omega)} = 1$. Our main theorem computes $G^{(\omega)}_{\Q}$ for all one-relator groups and generalises Baumslag's result.

\begin{theorem}
\label{main}
    Let $F$ be a free group, $w\in F$ a word and $G = F/\normal{w}$ a one-relator group. There is a word $r\in F$, of length bounded above by that of $w$, and an integer $k\geqslant 1$ such that $w\in r^k[\normal{r}, \normal{r}]$ and
    \[
    G_{\Q}^{(\omega+1)} = G_{\Q}^{(\omega)} = \normal{r}_G.
    \]
    In particular, the maximal residually rationally solvable quotient of $G$ is the one-relator group $F/\normal{r}$.
\end{theorem}

Torsion-free one-relator groups cannot contain finitely generated rationally perfect subgroups since they are locally indicable by a result of Brodskii \cite{Br84} (see also Howie \cite{Ho82}) and so the subgroup $G^{(\omega)}_{\Q}$ from \cref{main} is always infinitely generated when it is non-trivial. The prototypical examples of non residually (rationally) solvable one-relator groups are the Baumslag--Gersten groups \cite{Ba69}:
\[
G(1, n) \isom \langle a, t \mid [a^t, a] = a^{n-1}\rangle.
\]
Here we adopt the convention that $a^t = t^{-1}at$ and $[a^t, a] = a^ta(a^t)^{-1}a^{-1}$. When $n\neq 1$, the relator expresses (a power of) $a$ as a commutator over itself and a conjugate of itself. Thus, the normal closure of $a$ is perfect when $n = 2$ and is rationally perfect (has torsion abelianisation) otherwise. \cref{main} says that the maximal rationally perfect subgroup of any one-relator group is essentially always of this form; that is, the normal closure of a single element that can be read off from the relator. In fact, it is worth pointing out that in general, if $G = F/\normal{w}$ is a one-relator group, $k\geqslant 1$ and $w\in r^k[\normal{r}, \normal{r}]$, then $\normal{r}_G$ is a $\Q$-perfect subgroup of $G$.

If $G$ is a one-relator group with $H^2(G) = 0$, then a result of Strebel \cite{St74} implies that $G^{(\alpha)} = G^{(\alpha)}_{\Q}$ for all ordinals $\alpha$. Thus, \cref{main} is also true for the ordinary derived series of one-relator groups $G$ under the additional hypothesis that $H^2(G) = 0$. We thank Jens Harlander for pointing this out to us. We conjecture that \cref{main} holds also for the derived series of all one-relator groups, with a minor modification, see \cref{conjecture}.

With little work, we obtain a satisfying characterisation of residually $\Q$-solvable one-relator groups.

\begin{corollary}
\label{main_corollary}
Let $F$ be a free group, $w\in F$ a word and $G = F/\normal{w}$ a one-relator group. The following are equivalent:
\begin{enumerate}
\item\label{itm1} $G$ is residually $\Q$-solvable.
\item\label{itm2} $G^{(n)}_{\Q}$ is a free group for some integer $n\geqslant 0$.
\item\label{itm3} For all words $r\in F$ of length at most that of $w$, and all integers $k\geqslant 1$, if $w\in r^k[\normal{r}, \normal{r}]$, then $w$ is conjugate to $r$ or $r^{-1}$.
\end{enumerate}
\end{corollary}

The equivalence between (\ref{itm1}) and (\ref{itm2}) is a consequence of \cref{free_kernel}, a key step in the proof of \cref{main}. In fact, a consequence of our proofs is that the integer $n$ from (\ref{itm2}) may be taken to be the length of the relator $w$. If $G$ is not residually rationally solvable, then $G^{(n)}_{\Q}$ is an extension of the maximal rationally perfect subgroup by a free group.

Using \cref{main}, we recover and strengthen Baumslag's result.

\begin{corollary}
\label{positive}
Torsion-free positive one-relator groups are free-by-($\Q$-solvable) and so are residually $\Q$-solvable. Positive one-relator groups with torsion are free-by-solvable and so are residually solvable.
\end{corollary}

Using \cref{main}, we may also algorithmically determine whether a one-relator group is residually $\Q$-solvable or not. In fact, the algorithm computes the element $r$ from \cref{main}. Note that in general it is undecidable whether a finitely presented group is residually rationally solvable. 

\begin{corollary}
\label{algorithm}
There is an algorithm that, on input a word $w$ over an alphabet $S\sqcup S^{-1}$, decides whether the one-relator group $F(S)/\normal{w}$ is residually $\Q$-solvable.
\end{corollary}

It would be interesting to obtain similar characterisations for residually solvable one-relator groups and residually (torsion-free) nilpotent one-relator groups. An intriguing question of Arzhantseva asks whether all one-relator groups are residually amenable \cite[Problem 18.6]{Ko18}. In light of \cref{main}, it would be especially interesting to know whether the Baumslag--Gersten groups are residually amenable.

We conclude this introduction by remarking that this article is almost entirely self-contained: the only results we use and do not prove are old theorems of Magnus on one-relator groups \cite{Ma30,Ma32} and on residual $\Q$-solvability of free groups \cite{Ma35} and the Burns--Hale Theorem \cite{BH72}. We also take inspiration from Howie's tower method \cite{Ho81}.

\subsection*{Acknowledgements}
The author would like to thank Goulnara Arzhantseva, Jens Harlander and Andrei Jaikin-Zapirain for useful comments on a previous version of this article. The author would also like to thank the anonymous referees for their helpful suggestions. This work has received funding from the European Research Council (ERC) under the European Union's Horizon 2020 research and innovation programme (Grant agreement No. 850930).

\section{Weakly reducible complexes and their regular covers}

Throughout this article, a \emph{2-complex} will be taken to mean a two-dimensional CW-complex whose attaching maps are given by immersions. All maps between 2-complexes will be assumed to be \emph{combinatorial}, that is, open $n$-cells are sent homeomorphically to open $n$-cells. In fact, most maps will just be covering spaces or inclusions.

A \emph{weak elementary reduction} is a pair of 2-complexes $X\subset Z$ such that $Z - X$ consists of a single 1-cell and at most one 2-cell whose attaching map traverses the 1-cell. It is an \emph{elementary collapse} if $Z - X$ contains a 2-cell whose attaching map traverses the 1-cell precisely once. Note that our definition of a weak elementary reduction is a weak version of Howie's notion of an elementary reduction from \cite{Ho82}.

If $\alpha$ is an ordinal number, say a 2-complex $Z$ is \emph{weakly $\alpha$-reducible to $X$} if there are subcomplexes $X_{\beta}\subset Z$ for each $\beta\leqslant\alpha$ such that $X_0 = X$, $X_{\alpha} = Z$, $X_{\beta-1}\subset X_{\beta}$ is a weak elementary reduction for each successor ordinal $\beta$ and $X_{\beta} = \bigcup_{\gamma<\beta}X_{\gamma}$ for each limit ordinal $\beta$. Say $Z$ is \emph{weakly $\alpha$-reducible} if it is weakly $\alpha$-reducible to its 0-skeleton. When $\alpha \leqslant \omega$, we shall just say that $Z$ is \emph{weakly reducible} or \emph{weakly reducible to $X$}. For locally finite 2-complexes, any weakly $\alpha$-reducible 2-complex (to a given subcomplex) is actually weakly reducible (to a given subcomplex).

\begin{lemma}
\label{reshuffle}
Let $\alpha$ be an ordinal and let $Z$ be a 2-complex weakly $\alpha$-reducible to $X\subset Z$. If $Z - X$ has countably many cells and if each 1-cell in $Z-X$ is traversed by finitely many attaching maps of 2-cells, then $Z$ is weakly reducible to $X$.
\end{lemma}

\begin{proof}
    For each ordinal $\beta\leqslant \alpha$, denote by $X_{\beta}\subset Z$ the associated subcomplex. We construct a directed graph $\Gamma$ as follows. The vertices of this graph are the successor ordinals $\beta\leqslant \alpha$ and two vertices $\beta$ and $\gamma$ are connected by a directed edge if the following holds:
    \begin{enumerate}
        \item $\beta<\gamma$.
        \item $X_{\gamma} - X_{\gamma -1}$ contains a 2-cell $c$.
        \item $X_{\beta} - X_{\beta-1}$ contains a 1-cell traversed by the attaching map for $c$.
    \end{enumerate} 
Intuitively, we are encoding the dependency between the different weak elementary reductions in the graph $\Gamma$. Since each 1-cell in $Z - X$ is traversed by finitely many attaching maps of 2-cells, this directed graph is locally finite. Denote by $S\subset V(\Gamma)$ the vertices with no incoming edges. If $v\in V(\Gamma)$ is a vertex, denote by $\Gamma_v\subset \Gamma$ the union of all directed paths in $\Gamma$ that end at $v$. The well-order on the vertices, together with the local finiteness of $\Gamma$, implies that $\Gamma_v$ is finite for all $v\in V(\Gamma)$ as each directed path ending at $v$ yields a descending chain. Indeed, if $\Gamma_v$ was infinite, then the local finiteness of $\Gamma$ implies that there is some edge leading into $v$ with origin $u$ such that $\Gamma_u$ is infinite. Repeating this, one may find an infinite descending chain, contradicting the well-ordering on the vertices. Furthermore, the well-ordering also implies that $\Gamma$ is acyclic. Denote by $\depth(v)$ the maximal length of a directed path in $\Gamma_v$, which is well-defined since $\Gamma$ is acyclic, and by $\roots(v)\subset S$ the set of vertices in $\Gamma_v$ that have no incoming edges.

    Choose any injection $\phi\colon S\to \N$. We extend $\phi$ to a map $\bar{\phi}\colon V(\Gamma)\to \N$ by defining 
    \[
    \bar{\phi}(v) = \min{\{k \mid \depth(v)\leqslant k, \roots(v)\subset \phi^{-1}(\{0, \ldots, k-1\})\}}.
    \]
    Since $\Gamma$ is locally finite, $\bar{\phi}$ is a finite to one map.

    For each successor ordinal $\beta\leqslant \alpha$, denote by $c_{\beta} = X_{\beta} - X_{\beta-1}$, which is either a 1-cell or a 1-cell and a 2-cell. Now define $Z_0 = X$ and
    \[
    Z_{l+1} = Z_l \cup \bigcup_{\beta\in \bar{\phi}^{-1}(l+1)}c_{\beta}.
    \]
    By definition of $\bar{\phi}$, we have that $Z_{l+1}$ actually weakly $k$-reduces to $Z_{l}$ where $k = |\bar{\phi}^{-1}(l+1)|<\infty$ and $Z = \bigcup_{i=0}^{\infty}Z_i$. Thus, $Z$ weakly reduces to $X$.
\end{proof}

Finite sheeted covers of weakly reducible complexes are almost never themselves weakly reducible. Howie took advantage of the fact that infinite cyclic covers of weakly reducible complexes are weakly reducible in \cite{Ho82}. We here generalise this fact to locally indicable covers.

\begin{proposition}
\label{locally_indicable_cover}
    Let $Z$ be a 2-complex weakly reducible to $X$, such that $Z - X$ contains countably many 2-cells and each 1-cell in $Z-X$ is traversed by finitely many attaching maps of 2-cells. If $p\colon Y\to Z$ is a regular cover with $\deck(p)$ locally indicable, then $Y$ is weakly reducible to $p^{-1}(X)$.
\end{proposition}

\begin{proof}
    Express $Z$ as a union of weak elementary reductions $X = X_0\subset X_1\subset X_2\subset \ldots \subset Z$. If for each $n$, we have that $p^{-1}(X_{n+1})$ is weakly reducible to $p^{-1}(X_n)$, then $Y$ will be weakly $\omega^2$-reducible to $p^{-1}(X)$. By \cref{reshuffle}, $Y$ will be weakly reducible to $p^{-1}(X)$. Thus it suffices to prove the result in the case that $X\subset Z$ is a weak elementary reduction. If $Z - X$ does not contain a 2-cell, then the result is clear, so now suppose that $Z - X = e\cup c$, where $e$ is a 1-cell and $c$ is a 2-cell. If $Z_1\subset Z$ denotes the closure of $c$ and $X_1 = X\cap Z_1$, then $X_1\subset Z_1$ is a weak elementary reduction and if $p^{-1}(Z_1)$ weakly reduces to $p^{-1}(X_1)$, then $Y$ weakly reduces to $p^{-1}(X)$. Thus, we may assume that $Z$ itself is the closure of a single 2-cell and $X\subset Z$ is a weak elementary reduction to a 1-subcomplex. 

    The proof now proceeds by induction on $c(Z) = |\lambda| - |Z^{(0)}|$, where $\lambda\colon S^1\to Z^{(1)}$ denotes the attaching map for $c$. Consider the induced epimorphism $\pi_1(Z)\to \deck(p) = G$. Since $Z$ is the closure of a single 2-cell, it is compact and hence $G$ is finitely generated. Since $G$ is finitely generated and locally indicable, either $G = 1$, or there is some epimorphism $\pi_1(Z)\to \Z$ factoring through an epimorphism $G\to \Z$. In the first case there is nothing to prove so, let us assume that we are in the second case. Let $K\to Z$ be the induced $\Z$-cover. By construction, $\pi_1(Y)\leqslant \ker(\pi_1(Z)\to \Z) = \pi_1(K)$ and so $Y\to Z$ factors as a composition of covers $Y\to K\to Z$. Denote by $q\colon Y\to K$. 

    If $e_0\subset K$ is a lift of $e$, denote by $e_i = t^i\cdot e_0$, where $t$ is a generator of the deck group of $K\to Z$. Denote by $D\subset K$ the closure of some lift of $c$. Then let $m$ be the minimal integer and $M$ the maximal such that $e_m, e_M\subset D$. Since $D$ surjects $Z$, we have $c(D)\leqslant c(Z)$. Since $c(D) = c(Z)$ if and only if $D\to Z$ is an isomorphism, it follows that $c(D)<c(Z)$. The pairs $D^{(1)} - e_M = L\subset D$ and $D^{(1)} - e_m = U\subset D$ are both weak elementary reductions.

Now consider the subcomplexes 
\begin{align*}
D_{-1} &= q(p^{-1}(X)),\\
D_0 &= D_{-1}\cup D,\\
D_{i+1} &= D_i \cup t^{-i-1}\cdot D\cup t^{i+1}\cdot D, \qquad \forall i\geqslant 0.
\end{align*}
By the inductive hypothesis, $q^{-1}(t^i\cdot D)$ weakly reduces to $q^{-1}(t^i\cdot L)$ and to $q^{-1}(t^i\cdot U)$ for all $i\in \Z$. Thus we see that $q^{-1}(D_{i+1})$ weakly reduces to $q^{-1}(D_i)$ for all $i\geqslant -1$. But then $q^{-1}(K) = Y$ weakly $\omega^2$-reduces to $q^{-1}(D_{-1}) = p^{-1}(X)$ and so $Y$ weakly reduces to $p^{-1}(X)$ by \cref{reshuffle} as claimed.
\end{proof}

\section{Separating elements in locally indicable quotients}

The following theorem is a classical result due to Weinbaum \cite{We72}.

\begin{theorem}[Weinbaum's Theorem]
\label{weinbaum}
If $F$ is a free group, $w\in F$ is a word and $G = F/\normal{w}$ is a one-relator group, then every proper non-empty subword of $w$ is non-trivial in $G$.
\end{theorem}

Howie generalised \cref{weinbaum} in \cite{Ho82} to one-relator products of locally indicable groups. We here show that if this property holds in a locally indicable quotient of an arbitrary one-relator product, then the kernel admits a particularly nice decomposition as a free product. This will yield consequences for the residual properties of one-relator groups.

\begin{theorem}
\label{free_kernel}
 Let $A$ and $B$ be countable groups, let $n\geqslant 1$ be an integer, let $w\in A*B$ be a cyclically reduced word of length at least two that is not a proper power and consider the one-relator product $G = \frac{A*B}{\normal{w^n}}$. If $\phi\colon G\to Q$ is a homomorphism to a locally indicable group such that $\phi(u)\neq 1$ for each proper non-empty subword $u$ of $w$, then $\ker(\phi)$ splits as a free product of a free group, subgroups of $A$, subgroups of $B$ and copies of $\Z/n\Z$.
\end{theorem}

\begin{proof}
Let $X_A, X_B$ be 2-complexes with $\pi_1(X_A) = A$ and $\pi_1(X_B) = B$, let $X = X_A\sqcup X_B$ and let $Z$ be the 2-complex obtained by attaching a 1-cell $e$ to $X$ connecting the two components and attaching a 2-cell $c$ along an attaching map $\lambda$ spelling out the word $w^n$. Since $w$ is cyclically reduced of length at least two, $\lambda$ traverses $e$ and so $X\subset Z$ is a weak elementary reduction with $\pi_1(Z) = G$. A homomorphism $\phi\colon G\to Q$ to a locally indicable group induces a regular cover $p\colon Y\to Z$ with $\deck(p) = \Ima(\phi)$. By \cref{locally_indicable_cover}, $Y$ is weakly reducible to $p^{-1}(X)$. Since $\phi(u)\neq 1$ for each proper non-empty subword of $w$, it follows that each lift $\tilde{\lambda}$ of $\lambda$ to $Y^{(1)}$ factors as
\[
\begin{tikzcd}
S^1 \arrow[r, "\gamma"'] \arrow[rr, "\tilde{\lambda}", bend left] & S^1 \arrow[r, "\tilde{\lambda}'"'] & Y^{(1)}
\end{tikzcd}
\]
where $\gamma$ is an $n$-fold cover and where $\tilde{\lambda}'$ traverses each lift of $e$ at most once. Hence, when $n = 1$, \cref{locally_indicable_cover} implies that $Y$ is collapsible to a connected subcomplex consisting of $p^{-1}(X)$ and some lifts of $e$. In particular, $Y$ deformation retracts onto a (connected) subcomplex of $Y$ which decomposes as a union of a 1-subcomplex and the components of $p^{-1}(X)$, whence the conclusion of the theorem. Denote by $C_n$ the 2-complex with 1-skeleton $S^1$ and with 2-cell attached along an $n$-fold cover $S^1\to S^1$. If $n\geqslant 2$, then $Y$ is homotopy equivalent to a (connected) subcomplex of $Y$ which decomposes as a union of a 1-subcomplex and the components of $p^{-1}(X)$, wedged with copies of the 2-complexes $C_n$.
\end{proof}

If $\mathcal{C}$ is any collection of groups, we say that $G$ is residually $\mathcal{C}$ if for every non-trivial element $g\in G$, there is an epimorphism $\phi\colon G\to Q$ with $Q\in \mathcal{C}$ such that $\phi(g) \neq 1$. We say $G$ is \emph{fully residually $\mathcal{C}$} if for every finite set of non-trivial elements $S\subset G$, there is an epimorphism $\phi\colon G\to Q$ with $Q\in \mathcal{C}$ such that $1\notin \phi(S)$. When $\mathcal{C}$ is closed under taking subdirect products, the two notions coincide.

\begin{corollary}
\label{freeby}
Let $G$ be a non-free one-relator group and let $\mathcal{C}$ be any collection of locally indicable groups. If $G$ is fully residually $\mathcal{C}$, then $G$ has a surjective homomorphism to a group in $\mathcal{C}$ whose kernel is free.
\end{corollary}

\begin{proof}
Firstly, since each group in $\mathcal{C}$ is torsion-free, $G$ must also be torsion-free as elements of finite order lie in the kernel of any homomorphism to a torsion-free group. Let $F$ be a free group and $w\in F$ such that $G\cong F/\normal{w}$. Since $G$ is torsion-free and not free, we have that $w$ is an imprimitve element in some free factor of $F$ of rank at least two. Thus, we may express $F = A*B$ with $A, B\neq 1$ and with $w$ being a (after possibly replacing with a conjugate) cyclically reduced word over $A*B$ of length at least two that is also not a proper power. By \cref{weinbaum}, each proper non-empty subword of $w$ is non-trivial in $G$. Since $G$ is fully residually $\mathcal{C}$, there is some group $C\in \mathcal{C}$ and an epimorphism $\phi\colon G\to C$ such that $\phi(u)\neq 1$ for each proper non-empty subword $u$ of $w$. Now we may apply \cref{free_kernel} to obtain that $\ker(\phi)$ is a free group as required.
\end{proof}

\cref{freeby} applies for instance to residually torsion-free nilpotent, residually $\Q$-solvable and fully residually free one-relator groups. Since RFRS (which stands for residually finite rationally solvable) groups are (fully) residually virtually abelian and locally indicable \cite[Theorem 6.3]{OS24}, a fact that Okun--Schreve attribute to Fisher and Kielak, \cref{freeby} also implies that RFRS one-relator groups are free-by-\{virtually abelian locally indicable\}.

\section{The rational derived series of a one-relator group}

We now refine \cref{freeby} further in the case of residually rationally solvable groups. Let $G$ be a group. If $R$ is any ring, denote by
\[
G^{(i)}_{R} = \ker\left(G^{(i-1)}_{R}\to H_1\left(G^{(i-1)}_{R}, R\right)\right)
\]
the $i^{\text{th}}$ term of the \emph{$R$-derived series}, where $G^{(0)}_{R} = G$. A group $G$ is $R$-solvable if $G_{R}^{(n)} = 1$ for some $n$ and is residually $R$-solvable if $\bigcap_{i=0}^{\infty}G_{R}^{(i)} = 1$. Note that $G$ is residually $R$-solvable then it is also fully residually $R$-solvable since being $R$-solvable is closed under subdirect products. The derived series is the $\Z$-derived series and the rational derived series is the $\Q$-derived series defined above, a fact we shall use often in the proof of \cref{main}. Before embarking on the proof, we first deduce the consequences from the introduction.

\begin{proof}[Proof of \cref{main_corollary}]
(\ref{itm2}) implies (\ref{itm1}) since free groups are residually $\Q$-solvable (in fact, they are residually torsion-free nilpotent by a result of Magnus \cite{Ma35}) and a \{residually $\Q$-solvable\}-by-\{$\Q$-solvable\} group is residually $\Q$-solvable. (\ref{itm1}) implies (\ref{itm2}) by \cref{freeby}. (\ref{itm3}) implies (\ref{itm1}) by \cref{main}. Finally, if $G$ is residually $\Q$-solvable, then \cref{main} implies that for any $r\in F$ and $k\geqslant 1$ such that $w\in r^k[\normal{r}, \normal{r}]$, we must have that $\normal{r}_G = 1$ (since $\normal{r}_G$ is $\Q$-perfect) and so $\normal{w} = \normal{r}$. The fact that $w$ must then be conjugate to $r$ or $r^{-1}$ is a theorem of Magnus \cite{Ma30}.
\end{proof}

We may now prove the first part of \cref{positive}.

\begin{corollary}
Torsion-free positive one-relator groups are free-by-($\Q$-solvable) and so are residually $\Q$-solvable. 
\end{corollary}

\begin{proof}
Let $F$ be a free group, $S\subset F$ a free basis and $w\in F$ a cyclically reduced word that is positive over $S$. Let $r\in F$ be a cyclically reduced element and let $k\geqslant 1$ be such that $w \in r^k[\normal{r}, \normal{r}]$. Let $X$ be a presentation complex for $F/\normal{w}$. Let $Y\to X$ be the cover corresponding to the subgroup $\normal{r}\triangleleft F/\normal{w}=\pi_1(X)$ and let $\lambda\colon S^1\to Y$ be the cycle labelled by $w$ (at a chosen basepoint). By \cref{weinbaum}, the cycle $\rho\colon S^1\to Y$ labelled by $r$ is embedded. Let $C\subset Y$ be the image of $\rho$. By definition of $r$, we have that $[\lambda] = [\rho]^k$ as elements of $Z_1(Y, \Z)$. But since $w$ is positive, $[\lambda]$ is a sum of 1-cells, one for each 1-cell that $\lambda$ traverses, and each with coefficient $+1$. Since these 1-cells must be precisely the 1-cells in $C$, this implies $\lambda$ is supported in $C$. Since $w$ is not a proper power, this means that $k = 1$ and $w = r$. Now \cref{main_corollary} implies that $F/\normal{w}$ is residually $\Q$-solvable.
\end{proof}

Although one-relator groups with torsion cannot be residually $\Q$-solvable, we may use the canonical torsion-free quotient to establish residual solvability in some cases.

\begin{lemma}
\label{resid_solvable_torsion}
Let $F$ be a free group, let $w\in F$ be an element that is not a proper power and let $n\geqslant 2$ be an integer. If $F/\normal{w}$ is residually $\Q$-solvable, then $F/\normal{w^n}$ is residually solvable.
\end{lemma}

\begin{proof}
If $F/\normal{w}$ is residually $\Q$-solvable, there is a homomorphism $F/\normal{w}\to Q$ to a $\Q$-solvable group $Q$ such that each proper non-empty subword of $w$ is non-trivial in the image. Composing the epimorphism $F/\normal{w^n}\to F/\normal{w}$, we obtain a homomorphism $F/\normal{w^n}\to Q$ with kernel a free product of a free group and copies of $\Z/n\Z$ by \cref{free_kernel}. Since this kernel is residually solvable, $F/\normal{w^n}$ is residually solvable.
\end{proof}

Using \cref{resid_solvable_torsion} we may complete the proof of \cref{positive}.

\begin{corollary}
Positive one-relator groups are free-by-solvable and so are residually solvable.
\end{corollary}

The following is a slightly expanded version of \cref{algorithm} from the introduction.

\begin{corollary}
There is an algorithm that, on input a word $w$ over an alphabet $S\sqcup S^{-1}$, computes an element $r$ which normally generates $G_{\Q}^{(\omega)}$, where $G = F(S)/\normal{w}$, and moreover, decides whether $G$ is residually $\Q$-solvable.
\end{corollary}

\begin{proof}
We first show that given a word $r$ over $S$, we can decide whether there is some $k\geqslant 1$ such that $w\in r^k[\normal{r}, \normal{r}]$. The word problem for one-relator groups is decidable by a result of Magnus \cite{Ma32}. Hence, we can compute arbitrarily large balls in the Cayley graph for $F(S)/\normal{r}$. Computing the ball $B(|w|)$ of radius $|w|$ in the Cayley graph, to conclude the (sub)procedure we simply need to determine whether the chain in $C_1(B(|w|))$ determined by $w$ is equal to a multiple of the cycle determined by $r$ in $C_1(B(|w|))$, which is certainly decidable.

Now we enumerate all words $r$ of length bounded above by that of $w$ and use the above procedure to determine whether $w\in r^k[\normal{r}, \normal{r}]$ or not. For each such $r$, the subgroup $\normal{r}_G$ is $\Q$-perfect (as remarked in the introduction). Since the maximal $\Q$-perfect subgroup of $G$ is unique, there will be some such $r$ such that all others lie in $\normal{r}$. Since membership in $\normal{r}$ is decidable by \cite{Ma32}, we may find the $r$ that generates the maximal $\Q$-perfect subgroup of $G$. Finally, by \cref{main_corollary}, to determine whether $\normal{r}_G = 1$, we simply need to check whether $w$ is conjugate to $r$ or $r^{-1}$.
\end{proof}

We now turn to the proof of \cref{main}. We shall first need a simple, yet critical, lemma. Recall that if $R$ is a ring, the group ring $RG$ is the ring with elements the formal sums $r = \sum_{g\in G}r_g\cdot g$ where $r_g\in R$ and $r_g\neq 0$ for all but finitely many elements. The support of the element $r$ in the group ring $RG$ is defined as the set
\[
\supp(r) = \{g\in G \mid r_g \neq 0\}\subset G.
\]
Note that $\supp(rs)\subset \supp(r)\supp(s)$ for all $r, s\in RG$.

\begin{lemma}
\label{key_lemma}
Let $G$ be locally indicable, let $R$ be a domain and let $0\neq a, b\in RG$. Then $\supp(ab)\subset\supp(b)$ if and only if $\supp(a) = 1$.
\end{lemma}

\begin{proof}
Suppose for a contradiction that $\supp(a)\neq 1$, but that $\supp(ab)\subset \supp(b)$. If $1\notin \supp(a)$, then $\supp((a+1)b)\subset \supp(b)$, so by possibly replacing $a$ with $a+1$, we may assume that $1\in \supp(a)$. In particular, $|\supp(a)|\geqslant 2$. Since locally indicable groups are right orderable by a result of Burns--Hale \cite{BH72}, we may use \cite[Lemma 13.1.7]{Pa77} to conclude that there are two uniquely represented elements $g_1\cdot h_1, g_2\cdot h_2\in\supp(a)\cdot\supp(b)$ with $g_1\neq g_2$. If $g_1h_1\in \supp(b)$, then $g_1\cdot h_1 = 1\cdot(g_1h_1)$ implies that $g_1 = 1$. Similarly for $g_2h_2$. Since $g_1\neq g_2$, this implies that either $g_1h_1\notin\supp(b)$ or $g_2h_2\notin\supp(b)$. But this is a contradiction and so $\supp(a) = 1$.
\end{proof}

\begin{proof}[Proof of \cref{main}]
If $w = 1$ in $F$, then taking $r = 1$ the result holds since Magnus showed that free groups are residually $\Q$-solvable in \cite{Ma35}. So from now on suppose that $w\neq 1$. Let $S\subset F$ be a free generating set and let $X$ be a presentation complex for $G$ with a 1-cell for each $s\in S$ and a 2-cell spelling out $w$. Let 
\[
\ldots \to Y_n\to \ldots \to Y_1 \to Y_0 = X
\]
denote the regular covers with $\pi_1(Y_n) = G_{\Q}^{(n)}$ and denote by $N_n = \pi_1\left(Y_n^{(1)}\right)\triangleleft F$. The groups $G_n = F/N_n = G/G_{\Q}^{(n)}$ are $\Q$-solvable and so are each locally indicable. 

The maps $Y_{n+1}\to Y_{n}$ induce commuting maps of chain complexes:
\[
    \begin{tikzcd}
    0 \arrow[r] &C_2(Y_{n+1}, \Q) \arrow[r, "d_2^{n+1}"] \arrow[d] & C_1(Y_{n+1}, \Q) \arrow[r, "d_1^{n+1}"] \arrow[d] & C_0(Y_{n+1}, \Q) \arrow[r, "d_0^{n+1}"] \arrow[d] & \Q \arrow[r] \arrow[d] & 0 \\
    0 \arrow[r] & C_2(Y_{n}, \Q) \arrow[r, "d_2^{n}"] & C_1(Y_{n}, \Q) \arrow[r, "d_1^{n}"] & C_0(Y_{n}, \Q) \arrow[r, "d_0^{n}"] & \Q \arrow[r] & 0
    \end{tikzcd}
    \]
where each $C_i(Y_n, \Q)$ can be considered as a free $\Q G_n$-module with basis in correspondence with $G_n$-orbits of $i$-cells:
\begin{align*}
C_0(Y_n, \Q) &\isom \Q G_n\\
C_1(Y_n, \Q) &\isom \oplus_{s\in S}\Q G_n\\
C_2(Y_n, \Q) &\isom \Q G_n.
\end{align*}

Now denote by $\Gamma_n\subset Y_n^{(1)}$ the minimal (possibly empty) subcomplex such that $d^n_2(1)$ is supported in $C_1(\Gamma_n, \Q)\leqslant C_1(Y_n, \Q)$. Here $1$ is the unit in $\Q G_n \isom C_2(Y_n, \Q)$, or in other words, a 2-chain corresponding to a choice of a lift of the 2-cell in $X$ to $Y_n$. 

We have maps $C_1(\Gamma_{n+1}, \Q)\to C_1(Y_{n}, \Q)$, induced by the maps $C_1(Y_{n+1}, \Q)\to C_1(Y_{n}, \Q)$, such that the image of $C_1(\Gamma_{n+1}, \Q)$ contains $C_1(\Gamma_{n}, \Q)$. Thus, we have
\[
\dim_{\Q}C_1(\Gamma_{n}, \Q)\leqslant \dim_{\Q}C_1(\Gamma_{n+1}, \Q)\leqslant |\lambda|.
\]
where $\lambda\colon S^1\to X^{(1)}$ is the attaching map of the 2-cell. It follows that for some $m$, we have isomorphisms $C_1(\Gamma_n, \Q)\to C_1(\Gamma_m, \Q)$, and hence maps $\Gamma_n\to\Gamma_m$ for all $n\geqslant m$, induced by the covers $Y_n\to Y_m$. By possibly increasing $m$, we also have isomorphisms $C_0(\Gamma_n, \Q)\to C_0(\Gamma_m, \Q)$ and hence the maps $\Gamma_n\to \Gamma_m$ become isomorphisms for all $n\geqslant m$. Since $C_1(\Gamma_{m+1}, \Q)\to C_1(\Gamma_m, \Q)$ is an isomorphism and $\pi_1(Y_{m+1}) = \ker(\pi_1(Y_m)\to H_1(Y_m, \Q))$, it follows that the image of the cycles $Z_1(\Gamma_m, \Q)$ under the inclusion $Z_1(\Gamma_m, \Q)\to Z_1(Y_m, \Q)$ must lie in the boundaries $B_1(Y_m, \Q)$. 

We now claim that 
\[
\dim_{\Q}Z_1(\Gamma_m, \Q) \leqslant 1.
\]
Denote by $d_2^m(1) = \delta$ and by $\delta_s\in \Q G_m$ the projection of $\delta$ to the $\Q G_m$ factor corresponding to $s\in S$. If $\delta_s = 0$ for all $s\in S$, then $\Gamma_m = \emptyset$ and so the claim holds. Now suppose there is some $s\in S$ such that $\delta_s \neq 0$. If $\alpha\in \Q G_m$ such that $d_2^m(\alpha)\in Z_1(\Gamma_m, \Q)$, then $\alpha\cdot \delta\in Z_1(\Gamma_m, \Q)$. By definition of $\Gamma_m$, this implies that $\supp(\alpha\cdot\delta_s)\subset \supp(\delta_s)$. Since $G_m$ is locally indicable, by \cref{key_lemma} we have that $\supp(\alpha) = 1$ and so $\alpha\in \Q$. Thus, 
\[
Z_1(\Gamma_m, \Q)\cap B_1(Y_m, \Q) = Z_1(\Gamma_m, \Q) = \Q\cdot \delta.
\]
In particular, $\Gamma_m\isom S^1$. This completes the proof of the claim.

If $\Gamma_m = \emptyset$, set $\overline{X} = X^{(1)}$ and $r = 1\in F$. If $\Gamma_m\neq\emptyset$, denote by $\rho\colon S^1\to X^{(1)}$ the cycle (unique up to orientation) that lifts to an embedding in $\Gamma_m$ and let $\overline{X}$ be the 2-complex with the same 1-skeleton as $X$, but with a single 2-cell attached along $\rho$ instead of $\lambda$. Let $r\in F$ be the corresponding group element. We will now show that $r$ is the element required by the theorem.

By definition of $\rho$, for each $n\geqslant m$ there is a cover $\overline{Y}_n\to \overline{X}$ which on 1-skeleta is $Y_n^{(1)}\to X^{(1)}$. This identification induces $\Q G_n$-isomorphisms $C_i(Y_n, \Q)\to C_i(\overline{Y}_n, \Q)$ commuting with the boundary maps for $i\leqslant 1$. A lift of $\rho$ to $\Gamma_m$ yields a generator for $Z_1(\Gamma_m, \Z)$, but a lift of $\lambda$ to $Y_m$ only yields a generator of $Z_1(\Gamma_m, \Q)$. Note that since $\lambda$ is the attaching map of the 2-cell, there is a lift $\tilde{\lambda}\colon S^1\to Y_m$ such that the induced image of a generator of $Z_1(S^1, \Q)$ is precisely $d_2^m(1) \in Z_1(\Gamma_m, \Q)$. Thus, there is some integer $k\geqslant 1$ such that the $\Q G_n$-isomorphism $C_2(Y_n, \Q) \to C_2(\overline{Y}_n, \Q)$ defined by sending $1$ to $k$ times the 2-chain corresponding to the lift of the 2-cell supported in $\Gamma_n\subset \overline{Y}^{(1)}_n = Y_n^{(1)}$ commutes with the boundary maps:
\[
    \begin{tikzcd}
    0 \arrow[r] &C_2(Y_{n}, \Q) \arrow[r, "d_2^{n}"] \arrow[d, "\times k"] & C_1(Y_{n}, \Q) \arrow[r, "d_1^{n}"] \arrow[d] & C_0(Y_{n}, \Q) \arrow[r, "d_0^{n}"] \arrow[d] & \Q \arrow[r] \arrow[d] & 0 \\
    0 \arrow[r] & C_2(\overline{Y}_{n}, \Q) \arrow[r, "\overline{d}_2^{n}"] & C_1(\overline{Y}_{n}, \Q) \arrow[r, "\overline{d}_1^{n}"] & C_0(\overline{Y}_{n}, \Q) \arrow[r, "\overline{d}_0^{n}"] & \Q \arrow[r] & 0
    \end{tikzcd}
    \]
Note that when $r = 1$ we have that $Z_2(Y_n, \Q) = C_2(Y_n, \Q)$ and $C_2(\overline{Y}, \Q) = 0$. Hence, taking $k = 0$ the above diagram also commutes. By the definition of the rational derived group, the induced isomorphism $H_1(Y_n, \Q)\to H_1(\overline{Y}_n, \Q)$ from above implies that $\pi_1\left(\overline{Y}_{n+1}\right) = \pi_1\left(\overline{Y}_n\right)^{(1)}_{\Q}$ for all $n\geqslant m$. Hence, we have isomorphisms
\[
G_{m+i} = F/N_{m+i} = \pi_1(X)/\pi_1(Y_m)^{(i)}_{\Q} \isom \pi_1\left(\overline{X}\right)/\pi_1\left(\overline{Y}_{m}\right)_{\Q}^{(i)} = \pi_1\left(\overline{X}\right)/\pi_1\left(\overline{Y}_{m+i}\right)
\]
for all $i\geqslant 0$.

We now claim that $\pi_1\left(\overline{Y}_m\right)$ is free. If $r = 1$, this is clear since $\pi_1\left(\overline{X}\right)$ itself is free. If $r$ mentions only a single generator from $S$, then this is also clear since $\pi_1\left(\overline{X}\right)$ would also be free itself. Now suppose that $S = S_1\sqcup S_2$ and $r$ mentions a generator in $S_1$ and a generator in $S_2$. Since each lift of $\rho$ to $\overline{Y}_m$ is an embedding, we have that the image of each proper non-empty subword of $r$ under the homomorphism $(F(S_1)*F(S_2))/\normal{r} = \pi_1\left(\overline{X}\right)\to G_m = \deck\left(\overline{Y}_m\to\overline{X}\right)$ is non-trivial. Since $G_m$ is locally indicable, we may now apply \cref{free_kernel} to conclude that $\pi_1\left(\overline{Y}_m\right)$ is free. 

Since $\pi_1\left(\overline{Y}_m\right)$ is free, 
\[
\bigcap_{i=0}^{\infty}\pi_1\left(\overline{Y}_{m+i}\right) = \pi_1\left(\overline{Y}_{\omega}\right) = 1
\]
since free groups are residually $\Q$-solvable by a result of Magnus \cite{Ma35}. In particular, 
\[
\bigcap_{i=0}^{\infty}N_i = \pi_1\left(\overline{Y}_{\omega}^{(1)}\right) = \pi_1\left(Y_{\omega}^{(1)}\right) = \normal{r}.
\]
Finally, since the homology class of $w$ is equal to that of $r^k$ in $H_1\left(\normal{r}, \Z\right) \isom \normal{r}/[\normal{r}, \normal{r}]$, we have that $w\in r^k[\normal{r}, \normal{r}]$ as claimed.
\end{proof}

\begin{example}
Consider the one-relator group
\[
G = \langle a, b \mid [a[[a,b], [b, a^{-1}]], b]\rangle.
\]
We may re-express this relator as
\begin{align*}
[a[[a,b], [b, a^{-1}]], b] &= [a[[a,b], [a,b]^a], b]\\
					& = [a, b] \left([[a,b],[a,b]^a]^{ba^{-1}b^{-1}}\left([[a,b],[a,b]^a]^{-1}\right)^{a^{-1}b^{-1}}\right)
\end{align*}
and so $G_{\Q}^{(\omega)} = G_{\Q}^{(1)} = \normal{[a,b]}_G$. Hence, $G$ is not residually $\Q$-solvable and its abelianisation is its maximal $\Q$-solvable quotient. In fact, since $G_{\Q}^{(1)}$ is perfect, $G$ is not residually solvable either.
\end{example}

\subsection{The derived series}

We suspect the same results proved in this article should hold for the usual derived series of one-relator groups. Unfortunately, local indicability of the successive quotients was essential for our proofs. Following a suggestion of Jens Harlander, we may show that there is one case in which our result applies to the derived series.

\begin{corollary}
Let $G = F/\normal{w}$ be a one-relator group with $H^2(G) = 0$. There is a word $r\in F$ such that $w\in r[\normal{r}, \normal{r}]$ and
    \[
    G^{(\omega+1)} = G^{(\omega)} = \normal{r}_G.
    \]
In particular, the maximal residually solvable quotient of $G$ is the one-relator group $F/\normal{r}$.
\end{corollary}

\begin{proof}
If $G = F/\normal{w}$ is a one-relator group with $H^2(G) = 0$, we have that $H_2(G) = 0$ and $H_1(G)$ is torsion-free by Lyndon's identity theorem \cite{Ly50}. Hence, $G$ is an $E$-group in the sense of Strebel and so we may apply \cite[Theorem A]{St74} to conclude that $G^{(\alpha)} = G_{\Q}^{(\alpha)}$ for all ordinals $\alpha$. By \cref{main}, there is a word $r\in F$ and $k\geqslant 1$ such that $G^{(\omega)} = \normal{r}_G$ and such that $w\in r^k[\normal{r}, \normal{r}]$. Since $H_2(G) = 0$, we have that $r^k$ is non-trivial in the abelianisation of $F$. Since $H_1(G)$ is torsion-free, $r$ is not a proper power in the abelianisation of $H_1(G)$ and so $k = 1$.
\end{proof}

We conclude with a conjecture.

\begin{conjecture}
\label{conjecture}
Let $G = F/\normal{w}$ be a one-relator group. There is a word $r\in F$ such that $w\in r[\normal{r}, \normal{r}]$ and
    \[
    G^{(\omega+1)} = G^{(\omega)}  = \normal{r}_G.
    \]
In particular, the maximal residually solvable quotient of $G$ is the one-relator group $F/\normal{r}$.
\end{conjecture}

\bibliographystyle{amsalpha}
\bibliography{bibliography}

\end{document}